\documentclass{article}

\usepackage{a4wide}
\usepackage{amsmath,amssymb,amsthm}
\usepackage{graphicx}
\usepackage{hyperref}
\usepackage{url}

\newcommand{\ds}{\displaystyle}

\newcounter{theorem}
\newtheorem{Theorem}[theorem]{Theorem}

\newcounter{corollary}
\newtheorem{Corollary}[corollary]{Corollary}

\newcounter{proposition}
\newtheorem{Proposition}[proposition]{Proposition}

\newcounter{lemma}
\newtheorem{Lemma}[lemma]{Lemma}

\newcounter{example}
\newtheorem{Example}[example]{Example}

\newcounter{remark}
\newtheorem{Remark}[remark]{Remark}

\newcounter{definition}
\newtheorem{Definition}[definition]{Definition}

\begin{document}

\title{Regional Controllability and Minimum Energy Control\\ 
of Delayed Caputo Fractional-Order Linear Systems}

\author{Touria Karite$^{1}$ 
\and Adil Khazari$^{2}$     
\and Delfim F. M. Torres$^{3,}$\thanks{Correspondence: 
delfim@ua.pt; Tel.: +351-234-370-668 (D.F.M.T.)}} 


\date{$^{1}$Laboratory of Engineering, Systems and Applications, 
Department of Electrical Engineering \& Computer Science, 
National School of Applied Sciences, 
Sidi Mohamed Ben Abdellah University, 
Avenue My Abdallah Km 5 Route d'Imouzzer, 
Fez BP 72, Morocco; 
{\tt touria.karite@usmba.ac.ma}\\[0.3cm]
$^{2}$Laboratory of Analysis, Mathematics and Applications, 
National School of Commerce \& Management, 
Sidi Mohamed Ben Abdellah University, Fez, Morocco; 
{\tt adil.khazari@usmba.ac.ma}\\[0.3cm]
$^{3}$\mbox{Center for Research and Development in Mathematics and Applications (CIDMA),} 
Department of Mathematics, University of Aveiro, 3810-193 Aveiro, Portugal; 
{\tt delfim@ua.pt}\\[0.6cm]
{\small (In Memory of Prof. Dr. Jos\'{e} A. Tenreiro Machado)}}


\maketitle

\begin{abstract}
We study the regional controllability problem
for delayed fractional control systems through the use 
of the standard Caputo derivative. First, we recall 
several fundamental results and introduce the family 
of fractional-order systems under consideration. Afterwards, 
we formulate the notion of regional controllability for fractional 
systems with control delays and give some of their important properties.
Our main method consists in defining an attainable set, 
which allow us to prove exact and weak controllability.
Moreover, main results include not only those of controllability 
but also a powerful Hilbert uniqueness method that
allow us to solve the minimum energy optimal control control problem.
Precisely, an explicit control is obtained that drives the system 
from an initial given state to a desired regional state 
with minimum energy. Examples are given to illustrate 
the obtained theoretical results.

\medskip


\medskip

\end{abstract}


\section{Introduction}

The celebrated letter addressed by L'Hopital to Leibniz, 
wondering about the possibilities we can get when the order $n$ 
of the derivative is a fraction $1/2$, revolutionized Calculus 
and marked the birth of Fractional Calculus \cite{KM:BR:93}. 
Since its beginning, fractional calculus attracted many great mathematicians, 
who contributed, directly or indirectly, to its development \cite{MR2799292}. 
Today, many researchers consider fractional calculus as an important tool 
to solve different problems in various fields, e.g., physics, thermodynamics, 
chemistry, biology, classical and quantum mechanics, viscoelasticity, 
finance, engineering, signal and image processing, and automatics 
and control \cite{MR2761357,Sene:2019,Abuaisha:2019}. 

Let $\Omega$ be a bounded subset of $\mathbb{R}^{n}$ 
with a regular boundary $\partial{\Omega}$, 
the final time be $\tau>0$, $Q=\Omega\times[0,\tau]$, 
and $\Sigma=\partial\Omega\times[0,\tau]$. 
We then consider the system
\begin{eqnarray}
\label{sys1:eq1}
\begin{cases}
\ds { }_{0}^{C} \mathbb{D}_{t}^{r} z(x,t) 
= \mathcal{A}z(x,t)  + \mathfrak{Bu}(t-h), \quad & t\geq 0,\\
z(x,0)=z_{_{0}}(x),\\
\mathfrak{u}(t) = \varphi(t), \quad -h\leq t \leq 0,
\end{cases}
\end{eqnarray}
where ${ }_{0}^{C} \mathbb{D}_{t}^{r}$ denotes the left-sided
Caputo fractional order derivative of order $r\in(0,1)$ \cite{Book:IP:1999,Book:AAK:HMS:JJT:2006}. 
Note that $z$ is a function of two parameters but the derivative is an operator that acts on $t$.
The linear operator $\mathcal{A}$ is an infinitesimal generator 
of a $C_{0}$-semi-group $(\mathcal{T}(t))_{_{t\geq 0}}$ 
on the Hilbert state space $L^{2}(\Omega)$ \cite{Book:KJE:RN:2006,Book:MR:RCR:2004}. 
Here $h > 0$ is the time control delay and $\varphi(t)$ is the initial control function.
In the sequel, we have $z(x,t)\in \mathcal{Z} = L^{2}(0,\tau; L^{2}(\Omega))$ and the control 
$\mathfrak{u}\in \mathfrak{U} = L^{2}(0,\tau; \mathbb{R}^{p})$. The initial sate 
$z_{_{0}}\in L^{2}(\Omega)$, and the linear control operator 
$\mathfrak{B}: \mathbb{R}^{p}\rightarrow L^{2}(\Omega)$, which might be unbounded, 
depends on the number $p$ and the structure of actuators. 

The notion of controllability seeks to find a command or control that brings the system 
under study from an initial state to a desired final state. This is generally difficult 
to achieve, in particular for fractional order diffusion systems. This explains why a 
large number of scholars have been investigating control problems using the notion 
of "regional controllability". Such concept was first introduced by El Jai et al. 
in 1995 \cite{Jai:et:al:1995}, for parabolic systems, and then extended to the case 
of hyperbolic systems \cite{MR4401160}. The concept is widely used to investigate problems 
where the target of interest is not fully specified as a state, and relates only to a smaller 
internal region $\omega$ of the system domain $\Omega$, being crucial especially when it comes 
to real world problems, since the transfer cost is lower in the regional case, for instance 
in the case of wildfires, where the main purpose is to control it in a smaller region and 
one tries to minimize the costs 
\cite{Jai:Prit:1988:Halsted,Ge:Chen:Kou:grad:2016,Ge:Chen:Kou:bound:2016}.

In various processes, future states are dependent on current and previous states of the system, 
which implies that the models describing these processes should include delays, either in the 
state or in the control variables, or in both. If the delays are in the inputs, we are faced 
with systems with delayed commands. Due to the number of mathematical models describing dynamical 
systems with delays in the controls, solving controllability problems for such systems 
is of great importance. In particular, controllability problems for linear continuous-time fractional 
systems with a delayed control have been the subject of several works 
\cite{Kaczorek2011,Kaczorek:Rogowski2015,Ghasemi2022,Klamka2021}. However, 
it should be noted that the majority of the research in the area deals with the global case, 
that is, they treat controllability on the whole evolution domain. 
Here we are interested in studying the concept in a specific region $\omega\in\Omega$. 

Fractional delayed differential equations are equations involving fractional derivatives 
and delays. Unlike ordinary derivatives, they are nonlocal derivatives in nature and are able 
to model memory effects. Indeed, time delays express the history of a past state \cite{MR4475184}. 
Many real-world problems can be modeled in a more accurate way by including fractional derivatives 
and delays and in a specific subregion $\omega$ of the whole evolution domain of the system $\Omega$. 
For instance, when it comes to modeling several epidemiological problems, regional controllability 
of fractional delayed differential systems can be more plausible. In the case of monitoring glucose 
rate, fractional-order models provide a reasonable rate of movement of glucose from the blood 
into the environment \cite{Sakulrang2017,DeGaetano2021}.

Over the years, numerous mathematicians, utilizing their own notations and approaches, 
have defined different types of fractional derivatives and integrals. In this paper, 
we treat the controllability problem of a fractional diffusion equation in the sense 
of Caputo with a delay in the control. Recent works were elaborated to solve optimal 
control problems with delays by combining conformable and Caputo--Fabrizio fractional 
derivatives via artificial neural networks \cite{Kheyrinataj2022}. 
Here we define the regional controllability in the exact and weak senses; we give necessary 
and sufficient conditions under which the system is controllable; and we obtain the control 
that minimizes the energy cost functional.

The rest of this paper is structured as follows. 
Some definitions and fundamentals on fractional calculus 
are given in Section~\ref{sec2}. In Section~\ref{sec3}, 
a definition of regional fractional controllability for 
delayed systems is given and a necessary and sufficient condition 
to verify it is proved. Our main findings on controllability 
and optimal control are then formulated and proved in Section~\ref{sec4}. 
In Section~\ref{sec5}, we provide illustrative examples 
for both cases of a bounded and an unbounded control operator. 
We conclude with Section~\ref{sec6}, giving a summary of the main conclusions 
and some insightful open questions that still deserve in-depth investigations.


\section{Preliminary Results}
\label{sec2}

We begin with some definitions, properties, and known results about fractional calculus 
that will be used, in the sequel, to study system \eqref{sys1:eq1}. In particular, we 
recall the two more standard notions for fractional derivatives; 
the concept of solution to system \eqref{sys1:eq1}; 
and the fractional Green formula. In what follows
$g:\mathbb{R}^{+} \rightarrow \mathbb{R}$ is a given function.

\begin{Definition}[Caputo fractional derivatives, see, e.g., \cite{Book:IP:1999}]
The Caputo fractional derivative of order $r>0$ of 
a function $g : [0,\infty) \rightarrow \mathbb{R}$, is defined as
\begin{equation}
\label{Cap:Der:eq2}
{ }_{0}^{C} \mathbb{D}_{t}^{r} g(t) := { }_{0} I_{t}^{n-r} g^{(n)}(t) 
:= \frac{1}{\Gamma(n-r)} \int_{0}^{t}(t-\sigma)^{n-r-1} g^{(n)}(\sigma) d \sigma,
\end{equation}
where $n=-[-r]$, provided the right side is pointwise defined on $\mathbb{R}^{+}$, 
and ${ }_{0} I_{t}^{n-r}$ is the left-sided Riemann--Liouville fractional integral 
of order $n-r>0$ defined by 
\begin{equation}
\label{RL:Int:eq3}
{}_{0}I^{n-r}_{t}g(t) 
:= \ds\frac{1}{\Gamma(n-r)}
\int_{0}^{t}(t-\sigma)^{n-r-1}g(\sigma)d\sigma,\quad t>0.
\end{equation}
\end{Definition}

\begin{Definition}[Riemann--Liouville fractional derivatives, 
see, e.g., \cite{Paper:OPA:2004,Paper:GMB:2016,Paper:GMM:2011}]
The left-hand side and right-hand side Riemann--Liouville fractional 
derivatives of order $r$ of function $g$ are expressed by
\begin{equation}
\label{RL:Der:eq4}
{}_{0}\mathbb{D}^{r}_{t}g(t) 
:= \left(\frac{d}{dt}\right)^{n}{}_{0}I^{n-r}_{t}g(t) 
:=  \ds\frac{1}{\Gamma(n-r)} \frac{d^{n}}{dt^{n}}
\int_{0}^{t}(t-\sigma)^{n-r-1}g(\sigma)d\sigma, \quad t>0,
\end{equation}
and
\begin{equation}
\label{RL:Der:eq5}
{}_{t}\mathbb{D}^{r}_{\tau}g(t) := \ds\frac{1}{\Gamma(n-r)}\left( 
-\frac{d}{dt}\right)^{n} \int_{t}^{\tau}(\sigma-t)^{n-r-1}g(\sigma)d\sigma, 
\quad t<\tau,
\end{equation}
respectively, where $r\in(n-1,n)$, $n\in\mathbb{N}$.
\end{Definition}

\begin{Definition}[Mittag--Leffler function, see, e.g., \cite{Tomovski2022}]
The generalized Mittag--Leffler function is defined by
\begin{equation}
E_{r,s}(y) := \ds\sum_{i=1}^{\infty}\frac{y^{i}}{\Gamma(r i +s)},
\quad \operatorname{Re}(r)>0, 
\quad s, y \in \mathbb{C}. 
\end{equation}
\end{Definition}

\begin{Definition}[Three parameter Mittag--Leffler function \cite{Book:AAK:HMS:JJT:2006}]
The Prabhakar generalized Mittag--Leffler function is given by
\begin{equation}
E_{\alpha,\beta}^{\gamma}(y) := \ds\frac{1}{\Gamma(\gamma)}
\sum_{n=0}^{\infty}\frac{\Gamma(\gamma+n)y^{n}}{n!\Gamma(\alpha n+\beta)}, 
\quad \operatorname{Re}(\alpha) >0, 
\quad \alpha, \beta, \gamma \in\mathbb{C}.
\end{equation}
\end{Definition}

\begin{Definition}[See, e.g., \cite{Paper:FM:et:al:2007}]\label{def5}
For any given $F(x,t) \in \mathcal{Z}$, $0<r<1$, a function 
$z(x,t)\in \mathcal{Z}$ is said to be the general solution of 
\begin{equation*}
\begin{cases}
{}^{C}_{0}\mathbb{D}^{r}_{t}z(x,t) 
= \mathcal{A}z(x,t) + F(x,t),\qquad t\geq 0,\\
z(x,0) = z_{0}(x),
\end{cases}
\end{equation*}
and it is expressed by
$$
z(x,t) = \mathfrak{R}_{_{r}}(t)z_{_{0}}(x)
+ \ds\int_{0}^{t}(t-\sigma)^{r-1}
\mathcal{S}_{_{r}}(t-\sigma)F(\sigma)d\sigma,
$$
where 
\begin{equation}
\label{eq6}
\mathfrak{R}_{_{r}}(t) = \ds\int_{0}^{\infty}
\Phi_{_{r}}(\alpha)\mathcal{T}(t^{r}\alpha)d\alpha,
\end{equation}
and 
\begin{equation}
\label{eq7}
\mathcal{S}_{_{r}}(t) = r\ds\int_{0}^{\infty}\alpha
\Phi_{_{r}}(\alpha)\mathcal{T}(t^{r}\alpha)d\alpha.
\end{equation}
Here, $\{\mathcal{T}(t)\}_{t\geq 0}$ is the strongly 
continuous semigroup generated by operator $\mathcal{A}$,
$$
\Phi_{_{r}}(\alpha) = \ds\frac{1}{r}\alpha^{-1-1/r}\psi_{_{r}}(\alpha^{-1/r}),
$$
and $\psi_{_{r}}$ is the probability density function defined by
\begin{equation}
\label{eq:pdf}
\psi_{_{r}}(\alpha) = \ds\frac{1}{\pi} 
\sum_{n=1}^{\infty}(-1)^{n-1}\alpha^{-r n-1}
\frac{\Gamma(nr+1)}{n!}\sin(n\pi r),
\quad\alpha\in(0,\infty).
\end{equation}
\end{Definition}

\begin{Remark}
The density function given by \eqref{eq:pdf} 
is satisfying the properties
\begin{equation}
\label{density:eq8}
\ds\int_{0}^{\infty}e^{-\upsilon\alpha}\psi_{_{r}}(\alpha)d\alpha 
= e^{-\upsilon^{r}},\qquad \int_{0}^{\infty}\psi_{_{r}}(\alpha)d\alpha 
= 1,\quad r\in(0,1),
\end{equation}
and
\begin{equation}
\label{gamma:func}
\ds\int_{0}^{\infty}\alpha^{\nu}\Phi_{_{r}}(\alpha)d\alpha 
= \ds\frac{\Gamma(1+\nu)}{\Gamma(1+r\nu)},\quad \nu\geq 0.
\end{equation}
\end{Remark}	

The following hypotheses are used in our results:
\begin{itemize}
\item[] \textbf{($H_{1}$)} 
\quad The control operator $\mathfrak{B}$ is dense and $\mathfrak{B}^{*}$ exists;
\item[] \textbf{($H_{2}$)} 
\quad$\left(\mathfrak{B}\mathcal{S}_{r}(t)\right)^{*}$ exists and 
$\left(\mathfrak{B}\mathcal{S}_{r}(t)\right)^{*} = \mathcal{S}_{r}^{*}(t)\mathfrak{B}^{*}$.
\end{itemize}

Note that $(H_{1})$ and $(H_{2})$ hold when $\mathfrak{B}$ is bounded and linear. 
Throughout the paper, we use $z(x,t)$ for the state of the system.
Next we introduce the notion of mild solution of system \eqref{sys1:eq1},
using for it the notation $z_{_{\mathfrak{u}}}(x,t)$.
	
\begin{Definition}[Mild solution of system \eqref{sys1:eq1} \cite{Wei:2012}]
We say that a function $z_{_{\mathfrak{u}}}(x,t)\in \mathcal{Z}$ 
is a mild solution of system \eqref{sys1:eq1} if it satisfies
\begin{multline}
\label{MildSol:eq5}
z_{_{\mathfrak{u}}}(x,t)
= \mathfrak{R}_{_{r}}(t)z_{_{0}}(x)
+ \ds\int_{0}^{t-h}(t-\sigma-h)^{r-1}\mathcal{S}_{_{r}}(t-\sigma-h)
\mathfrak{B}\mathfrak{u}(\sigma)d\sigma \\
+ \ds\int_{-h}^{0}(t-\sigma-h)^{r-1}\mathcal{S}_{_{r}}(t-\sigma-h)
\mathfrak{B}\varphi(\sigma)d\sigma.
\end{multline}
\end{Definition}

We define $\mathcal{H} : L^{2}(0,\tau-h; \mathbb{R}^{p})\rightarrow L^{2}(\Omega)$ by
\begin{equation}
\label{H:eq10}
\mathcal{H}\mathfrak{u} = \ds\int_{0}^{\tau-h}(\tau-\sigma-h)^{r-1}
\mathcal{S}_{_{r}}(\tau-\sigma-h)\mathfrak{Bu}(\sigma)d\sigma,
\quad\text{for all}\quad \mathfrak{u}\in L^{2}(0,\tau-h; \mathbb{R}^{p}).
\end{equation}

Assume that $(H_{1})$--$(H_{2})$ hold and $(\mathcal{T}^{*}(t))_{_{t\geq 0}}$ 
is a semi-group generated by $\mathcal{A}^{*}$ on $L^{2}(\Omega)$, which 
is strong and continuous. For $\nu\in L^{2}(\Omega)$, we have
\begin{equation}
\label{H*:eq11}
\begin{array}{lll}
\left\langle \mathcal{H}\mathfrak{u} , \nu\right\rangle 
&= \left\langle\displaystyle\int_{0}^{\tau-h}(\tau-\sigma-h)^{r-1}
\mathcal{S}_{r}(\tau-\sigma-h)\mathfrak{Bu}(\sigma)d\sigma,\nu \right\rangle_{L^{2}(\Omega)}\\
& = \displaystyle\int_{0}^{\tau-h} \left\langle (\tau-\sigma-h)^{r-1}
\mathcal{S}_{r}(\tau-\sigma-h)\mathfrak{Bu}(\sigma),\nu\right\rangle_{L^{2}(\Omega)} d\sigma\\
& = \displaystyle\int_{0}^{\tau-h} \left\langle \mathfrak{u}(\sigma),
\mathfrak{B}^{*}(\tau-\sigma-h)^{r-1}
\mathcal{S}_{r}^{*}(\tau-\sigma-h)\nu\right\rangle_{\mathfrak{U}} d\sigma\\
& = \langle \mathfrak{u}, \mathcal{H}^{*}\nu \rangle,
\end{array}
\end{equation}
where $\langle\cdot,\cdot\rangle$ is the inner product on the vector space, and
\begin{equation*}
\mathcal{S}_{_{r}}^{^{*}}(t) = r\ds\int_{0}^{\infty}\alpha
\Phi_{_{r}}(\alpha)\mathcal{T}^{*}(t^{r}\alpha)d\alpha.
\end{equation*}
Then, one has
\begin{equation}
\label{H*:eq}
\mathcal{H}^{*}\nu = \mathfrak{B}^{*}(\tau-\sigma-h)^{r-1}
\mathcal{S}_{_{r}}^{*}(\tau-\sigma-h)\nu,
\quad\text{for all}\quad \nu\in L^{2}(\Omega).
\end{equation}
Let $\mathcal{H}$ be defined as in \eqref{H:eq10} and let us define 
the operator $\mathfrak{L}_{\varphi}$ in such a way that
\begin{equation}
\label{op-L}
\begin{array}{rll}
\mathfrak{L}_{\varphi} : L^{2}\left(-h,0;L^{2}(\Omega)\right) 
& \longrightarrow L^{2}(\Omega)\\
\varphi & \longmapsto \ds\int_{-h}^{0}(\tau-\sigma-h)^{r-1}
\mathcal{S}_{_{r}}(\tau-\sigma-h)\mathfrak{B}\varphi(\sigma)d\sigma.
\end{array}
\end{equation}
Following the same steps in the computation of $\mathcal{H}^{*}$, we get
$$
\mathfrak{L}_{\varphi}^{*} v 
= \mathfrak{B}^{*}(\tau-\sigma-h)^{r-1}
\mathcal{S}_{_{r}}^{*}(\tau-\sigma-h)v,
\quad\text{for all}\quad v\in L^{2}(\Omega).
$$

\begin{Remark}
Solutions of \eqref{sys1:eq1} are considered in the weak sense. Most often, 
they are noted $z(x,t;\mathfrak{u})$ and, to simplify, when there is no ambiguity, 
we use the short notation $z_{_{\mathfrak{u}}}(x,t)$.
\end{Remark}

The subsequent lemmas are necessary to demonstrate our main results:
Lemma~\ref{lem1} is used in the proof of our Theorem~\ref{thm:01};
while Lemma~\ref{lemma2} is useful to prove Theorem~\ref{thm02}.

\begin{Lemma}[See \cite{MR2579471,MR3571010}]
\label{lem1}
The operators $\mathfrak{R}_{r}(t)$ and $\mathcal{S}_{r}(t)$ are bounded and linear. 
Moreover, for every $z\in L^{2}(\Omega)$, we have
\begin{equation}
\|\mathfrak{R}_{r}(t)z\| \leq M\|z\|,
\qquad \text{and}\qquad \|\mathcal{S}_{r}(t)z\|
\leq \displaystyle\frac{r M}{\Gamma(1+r)}\|z\|.
\end{equation}
\end{Lemma}	

\begin{Lemma}[See, e.g., \cite{Book:KM:2009}]
\label{lemma1}
The operator $\mathcal{Q}$ on $[0,\tau]$, defined for a differentiable 
and integrable function $g$ by
\begin{equation}
\mathcal{Q}g(t):=g(\tau-t),
\end{equation}
is a reflection operator with the properties
\begin{equation}
\begin{array}{ll}
\mathcal{Q}{}_{0}I^{r}_{t}g(t) 
= {}_{t}I^{r}_{\tau}\mathcal{Q}g(t), 
\qquad \mathcal{Q}{}_{0}\mathbb{D}^{r}_{t}g(t) 
= {}_{t}\mathbb{D}^{r}_{\tau}\mathcal{Q}g(t),\\\\	
{}_{0}I^{r}_{t}\mathcal{Q}g(t) 
= \mathcal{Q}{}_{t}I^{r}_{\tau}g(t), 
\qquad {}_{0}\mathbb{D}^{r}_{t}\mathcal{Q}g(t) 
= \mathcal{Q}{}_{t}\mathbb{D}^{r}_{\tau}g(t).
\end{array}
\end{equation}
\end{Lemma}

In the next lemmas, we recall the integration by parts  
and the fractional Green formulas, that relate
left-sided Caputo derivatives with right-sided
Riemann--Liouville derivatives. 

\begin{Lemma}[Fractional integration by parts formula, 
see, e.g., \cite{conf:pap:IP:YQC:2007}]
\label{lemma2}
For $t\in [0,\tau]$ and $r\in(n-1,n)$, $n\in\mathbb{N}$, the 
integration by parts relation 
\begin{equation}
\displaystyle\int_{0}^{\tau}f(t){}^{C}_{0}\mathbb{D}^{r}_{t} g(t)dt 
= \sum_{i=0}^{n-1}(-1)^{n-1-i}\left[ g^{i}(t) 
{}_{t}\mathbb{D}^{r-1-i}_{\tau}f(t) \right]_{0}^{\tau} 
+ (-1)^{n}\int_{0}^{\tau}g(t){}_{t}\mathbb{D}^{r}_{\tau}f(t)dt
\end{equation}
holds.
\end{Lemma}

\begin{Remark}
If $0<r<1$, then we obtain from Lemma~\ref{lemma2} that
\begin{equation}
\displaystyle\int_{0}^{\tau}f(t) {}^{C}_{0}\mathbb{D}^{r}_{t} g(t)dt 
= \left[  g(t){}_{t}I^{1-r}_{\tau}f(t)\right]_{0}^{\tau} 
- \int_{0}^{\tau}g(t) {}_{t}\mathbb{D}^{r}_{\tau}f(t)dt.
\end{equation}
\end{Remark}

\begin{Lemma}[Fractional Green formula, see, e.g., 
\cite{Paper:GMB:2016,Paper:Bah:Tan:2018}]
\label{lema:FGF}
Let $0<r\leq 1$ and $t\in[0,\tau]$. Then,	
\begin{equation}
\begin{array}{lll}
\ds\int_{0}^{\tau} \int_{\Omega} 
& \left( {}^{C}_{0}\mathbb{D}^{r}_{t}z(x,t)
- \mathcal{A}z(x,t) \right)\phi(x,t) dx dt \\
&=  \ds\int_{0}^{\tau}\int_{\Omega} z(x,t) \left( - {}^{}_{0}\mathbb{D}^{r}_{\tau}\phi(x,t)
- \mathcal{A}^{*}\phi(x,t) \right)dx dt\\
&\quad + \ds\int_{\partial\Omega} z(x,\tau) {}_{t}I^{1-r}_{\tau}\phi(x,\tau)d\Gamma 
- \ds\int_{\partial\Omega}z(x,0){}_{t}I^{1-r}_{\tau}\phi(x,0)d\Gamma\\
&\quad + \ds\int_{0}^{\tau}\int_{\partial\Omega} z(x,t)
\frac{\partial \phi(x,t)}{\partial\nu_{\mathcal{A}}} d\Gamma dt
- \ds\int_{0}^{\tau}\int_{\partial\Omega} 
\frac{\partial z(x,t)}{\partial\nu_{\mathcal{A}}}\phi(x,t)d\Gamma,
\end{array}
\end{equation}
for any $\phi\in C^{\infty}(\overline{Q})$. 
\end{Lemma}

As a corollary of Lemma~\ref{lema:FGF}, 
the following result can be derived.

\begin{Corollary}
Let $0<r<1$. Then, for any $\phi \in \mathcal{C}^{\infty}(\overline{Q})$ 
such that $\phi(x, \tau)=0$ in $\Omega$ and $\phi=0$ on $\Sigma$, we get
\begin{equation}
\begin{array}{lll}
\ds\int_{0}^{\tau} \ds\int_{\Omega} 
& \left({}^{C}_{0}\mathbb{D}^{r}_{t}z(x,t)
- \mathcal{A} z(x, t)\right) \phi(x, t) \mathrm{d} x \mathrm{d} t\\
& = -\ds\int_{\Omega} z(x, 0) {}_{t}I^{1-r}_{\tau} \phi\left(x, 0\right) \mathrm{d} x
+\int_{0}^{\tau} \int_{\partial \Omega} z(x,t) \frac{\partial 
\phi(x,t)}{\partial v_{\mathcal{A}}} \mathrm{d} \Gamma \mathrm{d} t \\
&\quad +\ds\int_{0}^{\tau} \int_{\Omega} z(x, t)\left(- {}^{}_{0}\mathbb{D}^{r}_{\tau} 
\phi(x, t)- \mathcal{A}^{*} \phi(x, t)\right) \mathrm{d} x \mathrm{~d} t.
\end{array}
\end{equation}
\end{Corollary}	


\section{Regional Fractional Controllability}
\label{sec3}

Let $\omega$ be a given region and a subset of $\Omega$ with 
positive Lebesgue measure. The projection operator on $\omega$ 
is denoted by the restriction mapping
\begin{equation}
\label{rest-op}
\begin{array}{rll}
P_{_{\omega}} : L^{2}(\Omega) 
& \longrightarrow L^{2}(\omega)\\
y & \longmapsto P_{_{\omega}}y 
= y\arrowvert_{_{\omega}}.
\end{array}
\end{equation}

\begin{Definition}[Regional exact controllability at time $\tau$]
We say that system \eqref{sys1:eq1} is $\omega$-exactly 
controllable at time $\tau$ if, for any $z_{_{d}}\in L^{2}(\omega)$, 
there exists a control $\mathfrak{u}\in \mathfrak{U}$ such that
\begin{equation}
P_{_{\omega}}z_{_{\mathfrak{u}}}(x,\tau) = z_{_{d}}.
\end{equation}
\end{Definition}

\begin{Definition}[Regional weakly controllability at time $\tau$]
We say that system \eqref{sys1:eq1} is $\omega$-weakly  
controllable at time $\tau$ if, for every $z_{_{d}}\in L^{2}(\omega)$, 
given $\varepsilon >0$, there is a control $\mathfrak{u}\in \mathfrak{U}$ 
such that
\begin{equation}
\|P_{_{\omega}}z_{_{\mathfrak{u}}}(x,\tau)
-z_{_{d}}\|_{L^{2}(\omega)} \leq\varepsilon.
\end{equation}
\end{Definition}

\begin{Remark}
It is equivalent to say that system \eqref{sys1:eq1} is regionally exactly 
(resp. regionally weakly) controllable or to say that system \eqref{sys1:eq1} 
is $\omega$--exactly (resp. $\omega$--weakly) controllable.
\end{Remark}

Taking into account that system \eqref{sys1:eq1} is linear, 
for $\mathfrak{u}\in \mathfrak{U}$, let us consider the  
attainable set $\mathbb{A}(t)$ in $L^{2}(\Omega)$ defined by
\begin{equation}
\label{attainable:set}
\mathbb{A}(t) 
= \left\{ a(\cdot,t)\in L^{2}(\Omega) \,|\, 
a(x,t) = \mathcal{H}\mathfrak{u} + \mathfrak{L}_{\varphi} \varphi \right\},
\end{equation}
where operators $\mathcal{H}$ and $\mathfrak{L}_{\varphi}$ are defined 
by equations \eqref{H:eq10} and \eqref{op-L}, respectively.
The following result holds.

\begin{Theorem}[Necessary and sufficient conditions for regional controllability]
\label{thm:01}	
For any given $\tau>0$, system \eqref{sys1:eq1} is regionally exactly 
(resp. regionally approximately) controllable if, and only if, 
$$
P_{_{\omega}}\mathbb{A}(\tau) = L^{2}(\omega)
\qquad (\text{resp.} \,\, 
\overline{P_{_{\omega}}\mathbb{A}(\tau)} = L^{2}(\omega)).
$$
\end{Theorem}	

\begin{proof}
We prove the approximate controllability case. Using the same proof 
in \cite{Yamamoto2011,Yamamoto2014}, and assuming 
that $\mathfrak{u} \equiv 0$ for all $t$, 
system \eqref{sys1:eq1} admits a unique solution 
\begin{equation}
\label{solnocont}
\mathcal{X}(z_{0})(x,t)\in \mathcal{Z},
\qquad \text{such that}\qquad \mathcal{X}(z_{0})(x,t) 
= \mathfrak{R}_{r}(t)z_{0}(x).
\end{equation}
Then, using Lemma~\ref{lem1}, $\exists c>0$ satisfying 
$\|\mathcal{X}(z_{0})\|_{L^{2}(\Omega)} \leq c\|z_{0}\|_{L^{2}(\Omega)}$. 
Hence, \eqref{solnocont} is well defined. For every $\nu\in L^{2}(\omega)$, 
since $P_{_{\omega}}\mathcal{X}(z_{0})(\cdot,\tau)\in L^{2}(\omega)$, 
we obtain that
$$
\left(\nu - P_{_{\omega}}\mathcal{X}(z_{0})\right)(\cdot,\tau)
\in L^{2}(\omega),
$$
and
$$
\begin{array}{llll}
\|P_{\omega}z(x,\tau)-\nu\|_{L^{2}(\omega)} 
&= \|P_{_{\omega}}z(x,\tau) + P_{_{\omega}}\mathcal{X}(z_{0})(x,\tau) 
- \left( P_{_{\omega}}\mathcal{X}(z_{0})-\nu\right)(x,\tau)\|_{L^{2}(\omega)}\\
&= \|\left(P_{\omega}z - P_{\omega}\mathcal{X}(z_{0}) \right)(x,\tau)
- \left(\nu - P_{_{\omega}}\mathcal{X}(z_{0})\right)(x,\tau)\|_{L^{2}(\omega)}\\
&= \|P_{_{\omega}}a(x,\tau) - \left(\nu - P_{_{\omega}}\mathcal{X}(z_{0})
\right)(x,\tau)\|_{L^{2}(\omega)}.
\end{array}
$$
If $\overline{P_{\omega}\mathbb{A}(\tau)} = L^{2}(\omega)$, 
for any $\varepsilon > 0$, we can find $\mathfrak{u}\in \mathfrak{U}$ 
satisfying
$$
\left\| P_{\omega}a(\cdot,\tau) 
- \left(\nu - P_{\omega}\mathcal{X}(z_{0})\right)(\cdot,\tau) \right\|
\leq \varepsilon,
$$
where $a(\cdot,\cdot)$ is an element of the attainable set \eqref{attainable:set}.
This implies that $\left\| P_{\omega}z_{\mathfrak{u}}(\cdot,\tau) - \nu \right\| 
\leq \varepsilon$, where $z_{\mathfrak{u}}(\cdot,\tau) 
= \mathcal{X}(z_{0})(\cdot,\tau) + a(\cdot,\tau)$ is the mild solution 
of system \eqref{sys1:eq1}. Then system \eqref{sys1:eq1} is $\omega$-weakly 
controllable at time $\tau$.
	
On the other hand, for a given $\tau > 0$, system \eqref{sys1:eq1} 
is $\omega$-weakly controllable if for any $z_{d}\in L^{2}(\omega)$, 
given $\varepsilon>0$, there is a control $\mathfrak{u}\in \mathfrak{U}$ 
such that
$$
\begin{array}{lll}
\left\|P_{\omega}z_{_{\mathfrak{u}}}(x,\tau)-z_{d}\right\|_{L^{2}(\omega)} 
&= \left\| P_{\omega}z_{_{\mathfrak{u}}}(\cdot,\tau)
+ P_{\omega}\mathcal{X}(z_{0})(\cdot,\tau) 
- P_{\omega}\mathcal{X}(z_{0})(\cdot,\tau) - z_{d} \right\|_{L^{2}(\omega)}\\
&= \left\| P_{\omega}z_{_{\mathfrak{u}}}(\cdot,\tau) 
- P_{\omega}\mathcal{X}(z_{0})(\cdot,\tau) - \left( z_{d} 
-  P_{\omega}\mathcal{X}(z_{0})(\cdot,\tau)  \right)\right\|_{L^{2}(\omega)}\\
&= \left\| P_{\omega}a(\cdot,\tau) - \left(z_{d}
-P_{\omega}\mathcal{X}(z_{0})(\cdot,\tau)\right) \right\|_{L^{2}(\omega)}\\
&\leq\varepsilon.
\end{array}
$$
One has $\left(P_{\omega}z_{_{\mathfrak{u}}}(\cdot,\tau)
- P_{\omega}\mathcal{X}(z_{0})(\cdot,\tau)\right)\in P_{\omega}\mathbb{A}(\tau)$. 
Then, $\left(z_{d} - P_{\omega}\mathcal{X}(z_{0})(\cdot,\tau)\right)
\in L^{2}(\omega)$. Thus $\overline{P_{\omega}\mathbb{A}(\tau)} = L^{2}(\omega)$.
\end{proof}	

\begin{Proposition}
\label{prop-equ}
The following properties are equivalent
\begin{enumerate}
\item[(1)] system \eqref{sys1:eq1} is $\omega$-exactly controllable;
\item[(2)] $\operatorname{Im}\left(P_{\omega}(\mathcal{H}
+\mathfrak{L}_{\varphi})\right) = L^{2}(\omega)$;
\item[(3)] $\operatorname{Ker}(P_{\omega}) 
+ \operatorname{Im}\left(\mathcal{H}
+\mathfrak{L}_{\varphi}\right) = L^{2}(\Omega)$.
\end{enumerate}
\end{Proposition}	

\begin{proof}
(1) $\Leftrightarrow$ (2). Suppose that system \eqref{sys1:eq1} 
is $\omega$-exactly controllable. Then there exists $z_{d}\in L^{2}(\omega)$ 
such that $P_{\omega}z_{\mathfrak{u}}(x,\tau) = z_{d}$, which is equivalent to
$$
P_{\omega}\mathfrak{R}_{r}z_{0} + P_{\omega}\mathcal{H}\mathfrak{u} 
+ P_{\omega}\mathfrak{L}_{\varphi}\varphi = z_{d}.
$$
For $z_{0} = 0$, we have
$P_{\omega}\mathcal{H}\mathfrak{u} + P_{\omega}\mathfrak{L}_{\varphi}\varphi 
= z_{d} \Leftrightarrow \operatorname{Im}\left(P_{\omega}\mathcal{H}\right) 
+ \operatorname{Im}\left(P_{\omega}\mathfrak{L}_{\varphi}\right) = L^{2}(\omega)$.

(2) $\Rightarrow$ (3). For every $z\in L^{2}(\omega)$, we designate 
by $\widetilde{z}$ the prolongation of $z$ to $L^{2}(\Omega)$. Given
$\operatorname{Im}\left(P_{\omega}\mathcal{H}\right) 
+ \operatorname{Im}\left(P_{\omega}\mathfrak{L}_{\varphi}\right) = L^{2}(\omega)$, 
there is a control $\mathfrak{u}\in \mathfrak{U}$, 
$\varphi\in L^{2}\left(-h,0;L^{2}(\Omega)\right)$, 
and $z_{1}\in \operatorname{Ker}(P_{\omega})$, such that 
$\widetilde{z} = z_{1} + \mathcal{H}\mathfrak{u} + \mathfrak{L}_{\varphi}\varphi$.

(3) $\Rightarrow$ (2). For every $\widetilde{z}\in L^{2}(\Omega)$, 
it follows from (3) that $\widetilde{z} = z_{1} + z_{2} + z_{3}$, 
where $z_{1}\in\operatorname{Ker}(P_{\omega})$, 
$z_{2}\in\operatorname{Im}\mathcal{H}$, and $z_{3}\in\operatorname{Im}\mathfrak{L}$. 
Then, there exists $\mathfrak{u}\in \mathfrak{U}$ such that 
$\mathcal{H}\mathfrak{u} = z_{2}$ and $\varphi\in L^{2}(-h,0;L^{2}(\Omega))$ 
such that $\mathfrak{L}_{\varphi}\varphi = z_{3}$. 
Hence, it follows from \eqref{rest-op} that 
$$
\operatorname{Im}(P_{\omega}\mathcal{H}) 
+ \operatorname{Im}(P_{\omega}\mathfrak{L}_{\varphi}) = L^{2}(\omega).
$$
The proof is complete.		
\end{proof}	

\begin{Proposition}
The following properties are equivalent:
\begin{enumerate}
\item[(1)] system \eqref{sys1:eq1} is $\omega$-weakly controllable;
\item[(2)] $\overline{\operatorname{Im}\left(P_{\omega}(\mathcal{H}
+\mathfrak{L}_{\varphi})\right)} = L^{2}(\omega)$;
\item[(3)] $\operatorname{Ker}(P_{\omega}) 
+ \overline{\operatorname{Im}\left(\mathcal{H}
+\mathfrak{L}_{\varphi}\right)} = L^{2}(\Omega)$.
\end{enumerate}
\end{Proposition}
\begin{proof}
The proof that (1)$\Leftrightarrow$ (2)$\Leftrightarrow$ (3)
is similar to the proof of Proposition~\ref{prop-equ}
and is left to the reader.
\end{proof}		
We give an illustrative example of application of our results.
\begin{Example}
Consider the time fractional order differential system with a zonal actuator 
governed by the state equation
\begin{equation}
\label{ex1}
\left\{\begin{array}{l}
{}^{C}_{0}\mathbb{D}^{r}_{t} z(x, t) = \Delta z(x, t) 
+ P_{\left[\beta_{1}, \beta_{2}\right]} \mathfrak{u}(t-h)
\quad \text { in }[0,1] \times[0, \tau-h], \\
z(0) = z_{0},\\
\mathfrak{u}(t) = \varphi(t),
\end{array}\right.
\end{equation}
where $0<r<1$, $\mathfrak{Bu}=P_{\left[\beta_{1}, \beta_{2}\right]} \mathfrak{u}$ 
and $0 \leq \beta_{1} \leq \beta_{2} \leq 1$. Moreover, 
since $\mathcal{A}=\displaystyle\frac{\partial^{2}\cdot}{\partial x^{2}}$ 
is a self-adjoint operator, we get that the eigenvalues of the operator 
$\mathcal{A}$ are given by $\upsilon_{i}=-i^{2} \pi^{2}$ and its eigenfunctions 
by $\zeta_{i}(x)=\sqrt{2} \sin (i \pi x)$. 
The uniformly continuous semigroup 
generated by $\mathcal{A}$ is
$$
\Xi(t) z(x, t)=\sum_{i=1}^{\infty} e^{\left(\upsilon_{i} t\right)}
\left(z, \zeta_{i}\right)_{L^{2}(0,1)} \zeta_{i}(x).
$$
It implies
$$
\mathcal{S}_{r}(t) z(x,t)=r \int_{0}^{\infty} \theta \phi_{r}(\theta) 
\Xi\left(t^{r} \theta\right) z(x,t) d \theta
=r\sum_{i=1}^{\infty} E_{r, r+1}^{2}\left(\upsilon_{i} 
t^{r}\right)\left(z, \zeta_{i}\right)_{L^{2}(0,1)} \zeta_{i}(x),
$$
and we obtain that
\begin{equation}
\label{eqHLphi}
\begin{aligned}
\left(\mathcal{H} + \mathfrak{L}_{\varphi}\right)^{*} z(x,t)
&=2\left[\mathfrak{B}^{*}(\tau-\sigma-h)^{r-1} 
\mathcal{S}_{r}^{*}(\tau-\sigma-h) z\right](x,t)\\
&=2r\mathfrak{B}^{*}(\tau-\sigma-h)^{r-1} \sum_{i=1}^{\infty} 
E_{r, r+1}^{2}\left(\upsilon_{i}(\tau-\sigma-h)^{r}\right)
\left(z, \zeta_{i}\right)_{L^{2}(0,1)} \zeta_{i}(x) \\
&=2r(\tau-\sigma-h)^{r-1} \sum_{i=1}^{\infty} 
E_{r, r+1}^{2}\left(\upsilon_{i}(\tau-\sigma-h)^{r}\right)
\left(z, \zeta_{i}\right)_{L^{2}(0,1)} 
\int_{\beta_{1}}^{\beta_{2}} \zeta_{i}(x) d x,
\end{aligned}
\end{equation}
while from $\displaystyle\int_{\beta_{1}}^{\beta_{2}}\zeta_{i}(x) d x 
=\displaystyle\frac{\sqrt{2}}{i\pi}\sin\frac{i\pi(\beta_{1}
+\beta_{2})}{2}\sin\frac{i\pi(\beta_{1}-\beta_{2})}{2}$ we get 
$\operatorname{Ker}\left(\mathcal{H} + \mathfrak{L}_{\varphi}\right)^{*} \neq ~\{0\}$, 
i.e., $\operatorname{Im}\left(\mathcal{H} + \mathfrak{L}_{\varphi}\right)\neq L^{2}(\omega)$, 
which means that system \eqref{ex1} is not controllable on $\Omega = [0,1]$.  
Let $\left[\beta_{1} = 0,\beta_{2}=1/3\right]$. We then have
$$
\begin{array}{lll}
\displaystyle\int_{0}^{1/3} \zeta_{i}(x)dx 
&= \displaystyle\int_{0}^{1/3}\sqrt{2}\sin(i\pi x)dx \\
&= \sqrt{2}\left[-\displaystyle\frac{\cos(i\pi x)}{i\pi}\right]_{0}^{1/3}\\
&= \displaystyle\frac{\sqrt{2}}{i\pi}\left(1-\cos(i\pi/3)\right).
\end{array}
$$
If $\omega = [1/3,2/3]\subset[0,1]$, then
\begin{equation}
\label{eqomeg}
\begin{split}
\left(\mathcal{H}+\mathfrak{L}_{\varphi}\right)^{*}
P_{\omega}^{*}\left(P_{\omega}\zeta_{j}\right) 
& = 2r(\tau-\sigma-h)^{r-1}\displaystyle\sum_{k=1}^{\infty}
E_{r,r+1}^{2}\left(\upsilon_{k}(\tau-\sigma-h)^{r}\right)\langle\zeta_{i},
\zeta_{j}\rangle_{L^{2}(\omega)} \int_{0}^{1/3}\zeta_{i}(x)dx\\
&= \displaystyle\sum_{k=1}^{\infty}\frac{rE_{r,r+1}^{2}
\left(\upsilon_{k}(\tau-\sigma-h)^{r}\right)}{2(\tau-\sigma-h)^{1-r}}
\langle \zeta_{i},\zeta_{j}\rangle_{L^{2}(\omega)}
\int_{0}^{1/3}\zeta_{i}(x)dx\\
& = \displaystyle\sum_{k=1}^{\infty}\frac{rE_{r,r+1}^{2}
\left(\upsilon_{k}(\tau-\sigma-h)^{r}\right)}{2(\tau-\sigma-h)^{1-r}}
\int_{1/3}^{2/3}\zeta_{i}(x)\zeta_{j}(x)dx 
\frac{\sqrt{2}}{i\pi}(1-\cos(i\pi/2))\\
& = \displaystyle\sum_{k=1}^{\infty}\frac{rE_{r,r+1}^{2}\left(\upsilon_{k}
(\tau-\sigma-h)^{r}\right)}{\sqrt{2}i\pi(\tau-\sigma-h)^{1-r}}
\int_{1/3}^{2/3}\zeta_{i}(x)\zeta_{j}(x)dx\left[1-\cos(i\pi/2)\right]\\
& \neq 0.
\end{split}
\end{equation}
We conclude that the state $\zeta_{j}$ is reachable on $\omega$.
\end{Example}	


\section{Optimal Control with a Regional Target}
\label{sec4}

Fractional optimal control is a rapidly developing topic, see for instance
\cite{Mozyrska:Torres:2010,Mozyrska:Torres:2011,MR4493136,MR4457399}. 
This section is motivated by the results of 
\cite{Jai:et:al:1995,Jai:Pritch:1988,Zerrik:1993,MyID:2017:a,MyID:2018:a,MyID:2018:b},
being devoted to prove that the steering control is a minimizer 
of a suitable optimal control problem. For that we use an extended 
version of the Hilbert uniqueness method (HUM) first introduced 
by Lions \cite{Book:Lions:1971,Book:Lions:1988}.

Let $F$ be a closed subspace of $L^{2}(\Omega)$. Our extended optimal control problem 
consists in seeking a minimum-norm control that drives the system to $F$ 
at time $\tau$. Precisely, we consider
\begin{equation}
\label{sec4:min:pb}
\inf_{\mathfrak{u}} \mathcal{J}(\mathfrak{u}) 
= \inf_{\mathfrak{u}} \left\{ \ds\int_{0}^{\tau}
\frac{1}{2}\|\mathfrak{u}(t)\|^{^{2}}dt \quad 
: \quad \mathfrak{u}\in \mathfrak{U}_{_{ad}} \right\},
\end{equation}
where $\mathfrak{U}_{_{ad}}= \left\lbrace \mathfrak{u}\in \mathfrak{U}\quad|
\quad P_{\omega}z_{_{\mathfrak{u}}}(\cdot,\tau)-z_{_{d}}\in F\right\rbrace$, 
and the set
\begin{equation*}
F^{\circ} = \left\lbrace f\in L^{2}(\Omega)\quad|
\quad f=0\quad\mbox{in}\;\Omega\setminus\omega\right\rbrace.
\end{equation*}
For $\psi_{_{0}}\in F^{\circ}$, we consider the system 
\begin{equation}
\label{sec4:adjoint:sys}
\begin{cases}
\ds {}^{C}_{t}\mathbb{D}^{r}_{\tau}\mathcal{Q}\psi(t) 
= -\mathcal{A}^{*}\mathcal{Q}\psi(t),\\
\ds\lim_{t\rightarrow \tau^{-}}{}_{t}I^{1-r}_{\tau}\mathcal{Q}\psi(t) 
= \psi_{_{0}} \subset L^{2}(\Omega),
\end{cases}
\end{equation}
in $L^{2}(\Omega)$ and let
\begin{equation}
\label{sec4:semi-norm}
\|\psi_{_{0}}\|_{_{F^{\circ}}}^{^{2}} 
=\ds\int_{0}^{\tau}\|\mathfrak{B}^{*}(\tau-\sigma-h)^{r-1}
\mathcal{S}_{r}^{*}(\tau-\sigma-h)P_{_{\omega}}^{*}\psi(\sigma)\|^{2}d\sigma,
\end{equation}
which is a semi-norm on $F^{\circ}$.

Using Lemma~\ref{lemma1}, we can rewrite \eqref{sec4:adjoint:sys} as
\begin{equation}
\label{sec4:sys:ss:Q}
\begin{cases}
\ds {}^{C}_{0}\mathbb{D}^{r}_{t}\psi(t) 
= -\mathcal{A}^{*}\psi(t),\\
\ds\lim_{t\rightarrow 0^{+}}{}_{0}I^{1-r}_{t}\psi(t) 
= \psi_{_{0}} \subset L^{2}(\Omega),
\end{cases}
\end{equation}
with the solution given by $\psi(t) = -t^{r-1}K_{_{r}}^{*}\psi_{_{0}}$. 

\begin{Theorem}
\label{thm02}	
If $\mathfrak{u}$ spans $\mathfrak{U}$, 
then $\overline{z_{\mathfrak{u}}(x,\tau)}
= L^{2}(\Omega)$.
\end{Theorem}	

\begin{proof}
Take $z(x,0) = z_{0}(x) = 0$ 
and suppose that $z_{\mathfrak{u}}(x,\tau)$ 
is not dense in $L^{2}(\Omega)$. Consequently, there is 
$\psi_{_{0}}\in L^{2}(\Omega)$, $\psi_{_{0}}\neq 0$, such that
\begin{equation}
\label{prod-scal}
\langle z_{\mathfrak{u}}(x,\tau),\psi_{_{0}}\rangle = 0, 
\qquad\forall \mathfrak{u}\in \mathfrak{U}.
\end{equation}
Multiplying both sides of \eqref{sec4:adjoint:sys} by $z(t)$, 
then integrating in $Q$, we get
\begin{equation}
\label{eqq28}
\begin{array}{ll}
\ds\int_{\Omega}\int_{0}^{\tau}z(x,t){}^{C}_{0}\mathbb{D}^{r}_{t}
\mathcal{Q}\psi(t)dtdx &= \ds\int_{0}^{\tau}\langle z(x,t),
-\mathcal{A}^{*}\mathcal{Q}\psi(t)\rangle_{\Omega} dt\\
&= - \ds\int_{0}^{\tau}\langle \mathcal{A}z(x,t), 
\mathcal{Q}\psi(t)\rangle_{\Omega}dt. 
\end{array}
\end{equation}
From Lemma~\ref{lemma2}, we have
\begin{equation}
\label{eqq29}
\begin{array}{llll}
\ds\int_{\Omega}\int_{0}^{\tau}z(x,t){}^{C}_{0}\mathbb{D}^{r}_{t}\mathcal{Q}\psi(t)dt dx 
&= \ds\int_{\Omega}\left[z(x,t){}_{t}I^{1-r}_{\tau}\mathcal{Q}\psi(t)\right]_{0}^{\tau} 
- \int_{\Omega}\int_{0}^{\tau}\mathcal{Q}\psi(t){}^{C}_{0}\mathbb{D}^{r}_{t}z(x,t)dt dx\\
&= \left\langle z(x,\tau), \ds\lim_{t\rightarrow\tau}{}_{t}I^{1-r}_{\tau}
\mathcal{Q}\psi(t)\right\rangle_{\Omega} - \left\langle z(x,0), 
\ds\lim_{t\rightarrow 0}{}_{t}I^{1-r}_{\tau}\mathcal{Q}\psi(t) \right\rangle_{\Omega} \\
&\quad - \ds\int_{0}^{\tau}\left\langle \mathcal{Q}\psi(t),
{}^{C}_{0}\mathbb{D}^{r}_{t}z(x,t)\right\rangle_{\Omega} dt\\
&= \langle z(x,\tau), \psi_{_{0}}\rangle 
- \ds\int_{0}^{\tau}\left\langle 
\mathcal{Q}\psi(t), {}^{C}_{0}\mathbb{D}^{r}_{t} z(x,t)\right\rangle_{\Omega} dt\\
&= \langle z(x,\tau), \psi_{_{0}}\rangle - \ds\int_{0}^{\tau}\langle 
\mathcal{Q}\psi(t), \mathcal{A} z(x,t) + \mathfrak{Bu}(t-h)\rangle dt.
\end{array}
\end{equation}
From equations \eqref{eqq28} and \eqref{eqq29}, one has
\begin{equation}
\langle z(x,\tau), \psi_{_{0}}\rangle 
= \ds\int_{0}^{\tau}\langle \mathcal{Q}\psi(t), 
\mathfrak{Bu}(t-h)\rangle dt.
\end{equation}
Using \eqref{prod-scal}, we have 
\begin{equation}
\ds\int_{0}^{\tau} \langle \mathcal{Q}\psi(t), 
\mathfrak{Bu}(t-h)\rangle dt = 0 
\Leftrightarrow \mathcal{Q}\psi(t) = \psi(\tau-t) \equiv 0, 
\quad\text{in}\,\,\, L^{2}(\Omega), \, \forall t\in[0,\tau].
\end{equation}
Then $\psi_{_{0}} = 0$, which is a contradiction. The proof is complete.
\end{proof}	

To proceed with the HUM approach, we need first to prove that 
the semi-norm $||\cdot||$ on $F^{\circ}$
\eqref{sec4:semi-norm} is a norm. We prove the next result.

\begin{Lemma}\label{lem5}
Assuming that $(H_{1})$--$(H_{2})$ hold, \eqref{sec4:semi-norm} 
defines a norm of $F^{\circ}$ when system \eqref{sys1:eq1} 
is $\omega$-weakly controllable.
\end{Lemma}	

\begin{proof}
Suppose that system \eqref{sys1:eq1} is $\omega$-weakly controllable. 
Then $\operatorname{Ker}((\mathcal{H}+\mathfrak{L}_{\varphi})^{*}
P_{_{\omega}}^{*}) = \{0\}$, that is,
\begin{equation*}
\mathfrak{B}^{*}(\tau-\sigma-h)^{r-1}
\mathcal{S}_{_{r}}^{*}(\tau-\sigma-h)P_{_{\omega}}^{*}\psi = 0 
\Longrightarrow \psi = 0.
\end{equation*}
Therefore, for every $\psi_{_{0}}\in F^{\circ}$, it follows that
\begin{multline*}
\|\psi_{_{0}}\|_{_{F^{\circ}}} 
=\ds\int_{0}^{\tau}\|\mathfrak{B}^{*}(\tau-t-h)^{r-1}
\mathcal{S}_{r}^{*}(\tau-t-h)P_{_{\omega}}^{*}\psi(t)\|^{2}dt = 0 \\
\Leftrightarrow \mathfrak{B}^{*}(\tau-t-h)^{r-1}
\mathcal{S}_{r}^{*}(\tau-t-h)P_{_{\omega}}^{*}\psi(t) = 0.
\end{multline*}
Then \eqref{sec4:semi-norm} is a norm.
\end{proof}	
Furthermore, let us define an operator 
$\mathcal{M} : F^{\circ *} \rightarrow F^{\circ}$ by
\begin{equation}
\mathcal{M}f = \mathcal{P}(\phi(\tau)),
\end{equation}
where $\mathcal{P} = P_{_{\omega}}^{*}P_{_{\omega}}$ 
and $\phi$ is defined by
\begin{equation}
\begin{cases}
\ds {}^{C}_{0}\mathbb{D}^{r}_{t}\phi(t) 
= \mathcal{A}\phi(t) + \mathfrak{B}\mathfrak{B}^{*}(\tau-t-h)^{r-1}
\mathcal{S}_{_{r}}^{*}(\tau-t-h)\psi(t), \\
\phi (0) = z_{_{0}}.
\end{cases}
\end{equation}
We then decompose $\mathcal{M}$ as
$$
\mathcal{M}f = \mathcal{P}\left(\phi_{_{1}}(\tau) 
+ \phi_{_{2}}(\tau)\right),
$$
where
\begin{equation}
\begin{cases}
\ds {}^{C}_{0}\mathbb{D}^{r}_{t}\phi_{_{1}}(t) 
= \mathcal{A}\phi_{_{1}}(t), \\
\phi_{_{1}} (0) = z_{_{0}},
\end{cases}
\end{equation}
and 
\begin{equation}
\begin{cases}
\ds {}^{C}_{0}\mathbb{D}^{r}_{t}\phi_{_{2}}(t) 
= \mathcal{A}\phi_{_{2}}(t) + \mathfrak{B}\mathfrak{B}^{*}(\tau-t-h)^{r-1}
\mathcal{S}_{_{r}}^{*}(\tau-t-h)\psi(t), \quad t\in[0,\tau-h],\\
\phi_{_{2}} (0) = 0.
\end{cases}
\end{equation}
Let 
\begin{equation}
\label{eqisomo}
\Lambda \psi_{_{0}} = P_{_{\omega}}\mathcal{P}\left(\phi_{_{2}}(\tau)\right).
\end{equation}
Then the regional controllability problem leads to the resolution of equation
\begin{equation}
\label{eq-reg-cont-pb}
\Lambda \psi_{_{0}} = z_{d} - P_{_{\omega}}
\mathcal{P}\left(\phi_{_{1}}(\tau)\right).
\end{equation}
For any $f\in (F^{\circ})^{*}$ and $g\in F$, by Holder's inequality we have that
\begin{equation*}
\begin{array}{ll}
\langle \Lambda f, g\rangle 
&= \ds\int_{\Omega}P_{_{\omega}}\int_{0}^{\tau}(\tau-\sigma-h)^{r-1}
\mathcal{S}_{_{r}}(\tau-\sigma-h)\mathfrak{B}\mathfrak{B}^{*}(\tau-\sigma-h)^{r-1}\\
&\qquad \times \mathcal{S}_{_{r}}^{*}(\tau-\sigma-h)
P_{_{\omega}}^{*}f(\sigma)d\sigma g(x) dx\\
& \leq \ds\int_{0}^{\tau} \parallel \mathfrak{B}^{*}(\tau-\sigma-h)^{r-1}
\mathcal{S}_{_{r}}^{*}(\tau-\sigma-h)P_{_{\omega}}^{*}
f(\sigma)\parallel^{2}d\sigma \parallel g\parallel\\
&\leq \parallel f\parallel\parallel  g\parallel,
\end{array}
\end{equation*}
and  $\parallel \Lambda f\parallel \leq \parallel f\parallel$. 
Moreover, for any $f \in (F^{\circ})^{*}$, we obtain
\begin{equation*}
\begin{array}{ll}
&\langle \Lambda \psi_{_{0}}, \psi_{_{0}}\rangle_{F^{\circ *}, F^{\circ}} 
= \langle P_{_{\omega}}\mathcal{P}(\phi_{_{2}}(\tau)),\psi_{_{0}}\rangle\\ 
&= \ds\left\langle\int_{0}^{\tau}P_{\omega}(\tau-\sigma-h)^{r-1}
\mathcal{S}_{_{r}}(\tau-\sigma-h)\mathfrak{B}\mathfrak{B}^{*}(\tau-\sigma-h)^{r-1}
\mathcal{S}_{_{r}}^{*}(\tau-\sigma-h)P_{_{\omega}}^{*}\psi_{_{0}}(\sigma)d\sigma, 
\psi_{_{0}}(\sigma) \right\rangle\\
& = \ds\int_{0}^{\tau}\left\langle \mathfrak{B}^{*}(\tau-\sigma-h)^{r-1}
\mathcal{S}_{_{r}}^{*}(\tau-\sigma-h)P_{_{\omega}}^{*}\psi_{_{0}}(\sigma),
\mathfrak{B}^{*}(\tau-\sigma-h)^{r-1}\mathcal{S}_{r}(\tau-\sigma-h)
P_{\omega}^{*}\psi_{_{0}}(\sigma)\right\rangle d\sigma\\
& = \ds\int_{0}^{\tau}\| \mathfrak{B}^{*}(\tau-\sigma-h)^{r-1}
\mathcal{S}_{_{r}}^{*}(\tau-\sigma-h)
P_{_{\omega}}^{*}\psi_{_{0}}(\sigma) \|^{2}d\sigma \\
& = \|\psi_{_{0}}\|_{F^{\circ}}^{2}.
\end{array}
\end{equation*}
Consequently, if system \eqref{sys1:eq1} is $\omega$-weakly controllable at $\tau$, 
we obtain that $\psi_{_{0}} = 0$. From the uniqueness of $\mathcal{M}$, 
we get that $\Lambda$ defined in \eqref{eqisomo} is an isomorphism. 

\begin{Theorem}
\label{main-res}
If system \eqref{sys1:eq1} is $\omega$-weakly controllable, 
then, for any $z_{d} \in L^2(\omega)$, \eqref{eq-reg-cont-pb} 
has a unique solution $\psi_{_{0}} \in F^{\circ}$ and the control
\begin{equation}
\mathfrak{u}^{*}(t) 
= 
\begin{cases}
\mathfrak{B}^{*}(\tau-t-h)^{r-1}
\mathcal{S}_{r}^{*}(\tau-t-h)P_{\omega}^{*}\psi(t), 
\qquad 0\leq t\leq \tau-h,\\
\varphi(t),  
\qquad\qquad\qquad \qquad\qquad \qquad\qquad \qquad\, 
\tau-h\leq t\leq \tau,
\end{cases}
\end{equation}
steers the system to $z_{d}$ in $\omega$. Moreover, $\mathfrak{u}^{*}$ 
solves the minimization optimal control problem \eqref{sec4:min:pb}.
\end{Theorem}	

\begin{proof}
If system \eqref{sys1:eq1} is $\omega$-weakly controllable, 
then \eqref{sec4:semi-norm} is a norm. Let us consider the completion 
of $F^{\circ}$ regarding the norm \eqref{sec4:semi-norm} and let us denote 
it again by $F^{\circ}$. Now we prove that \eqref{eq-reg-cont-pb} 
admits a unique solution in $F^{\circ}$. For any $\psi_{_{0}}\in F^{\circ}$, 
one has
$$
\langle \Lambda \psi_{_{0}}, \psi_{_{0}}\rangle_{F^{\circ *},F^{\circ}} 
= \langle P_{\omega}\mathcal{P}\left(\phi_{_{1}}(\tau)\right), \psi_{_{0}}\rangle
= \|\psi_{_{0}}\|_{F^{\circ}}^{2}.
$$
Using Theorems~1.1 and 2.1 in \cite{Book:Lions:1971}, 
one can see that \eqref{eq-reg-cont-pb} has a unique solution.
Moreover, setting $\mathfrak{u} = \mathfrak{u}^{*}$ in system 
\eqref{sys1:eq1}, we have $P_{\omega}z_{\mathfrak{u}^{*}}(x,\tau) = z_{_{d}}$.
For $\mathfrak{u}$ and $v$ in $\mathfrak{U}_{ad}$, one has
$P_{\omega}z_{\mathfrak{u}}(x,\tau) = P_{\omega}z_{v}(x,\tau) = z_{d}$, 
so $P_{\omega}\left[z_{\mathfrak{u}}(x,\tau) - z_{v}(x,\tau) \right] = 0$.
We can easily find, for any $\psi_{_{0}}\in F^{\circ}$, that
$$
\begin{array}{ll}
&\langle \psi_{_{0}}, P_{_{\omega}}\left[z_{_{\mathfrak{u}}}(x,\tau)
-z_{_{v}}(x,\tau)\right]\rangle = 0\\ 
&\Leftrightarrow \langle \psi_{_{0}},  
P_{\omega}\ds\int_{0}^{\tau-h}(\tau-\sigma-h)^{r-1}
\mathcal{S}_{r}(\tau-\sigma-h)\mathfrak{B}\left[\mathfrak{u}(\sigma)
-v(\sigma)\right]d\sigma\rangle = 0\\
&\Leftrightarrow \ds\int_{0}^{\tau-h} 
\langle \mathfrak{B}^{*}(\tau-\sigma-h)^{r-1}
\mathcal{S}_{r}^{*}(\tau-\sigma-h)P_{_{\omega}}^{*}
\psi_{_{0}}, \mathfrak{u}(\sigma)-v(\sigma)\rangle d\sigma = 0,
\end{array}
$$
and
$$
\begin{array}{lll}
\mathcal{J}^{\prime}(\mathfrak{u})(\mathfrak{u}-v) 
&= \ds\int_{0}^{\tau-h}\langle \mathfrak{u}(\sigma), 
\mathfrak{u}(\sigma)-v(\sigma)\rangle d\sigma \\
&= \ds\int_{0}^{\tau-h}\langle \mathfrak{B}^{*}(\tau-\sigma-h)^{r-1}
\mathcal{S}_{r}^{*}(\tau-\sigma-h)P_{\omega}^{*}
\psi_{_{0}},\mathfrak{u}(\sigma)-v(\sigma)\rangle d\sigma\\
&= 0.
\end{array}
$$
Because $\mathfrak{U}_{ad}$ is convex, using Theorem 1.3 in  
\cite{Book:Lions:1971}, we establish 
the optimality of $\mathfrak{u}^{*}$.
\end{proof}	


\section{Examples}
\label{sec5}

We provide two illustrative examples, 
for both cases of a bounded (Section~\ref{sec:Ex1})
and an unbounded control operator (Section~\ref{sec:Ex2}). 


\subsection{Example~1: Zonal actuator}
\label{sec:Ex1}

We consider the system 
\begin{eqnarray}
\label{sys:exp}
\begin{cases}
\ds {}^{C}_{0}\mathbb{D}^{0.3}_{t}z(x,t)
= \Delta z(x,t)  + P_{[\beta_{1},\beta_{2}]}\mathfrak{u}(t-h), 
\quad & [0,1]\times[0,\tau-h],\\
z(x,0)=0,\\
\mathfrak{u}(t) = \varphi(t), \quad -h\leq t \leq 0,
\end{cases}
\end{eqnarray}
with fractional order $r = 0.3$.
Here the control operator $\mathfrak{B}$ 
is bounded, $[\beta_{1},\beta_{2}] = [0,1/2]$,
$\mathcal{A} = \Delta = \ds\frac{\partial^{2}\cdot}{\partial x^{2}}$, 
and the semigroup $\left(\mathcal{T}(t)\right)_{t\geq 0}$ is given by
$$
\mathcal{T}(t) z(x,t) = \ds\sum_{i=1}^{\infty}
e^{\upsilon_{i}t}(z,\zeta_{i})\zeta_{i}(x), 
\qquad x\in[0,1],
$$
where $\upsilon_{i} = -i^{2}\pi^{2}$ and 
$\zeta_{i}(x) = \sqrt{2}\sin(i\pi x)$, $i = 1,2,\ldots$
Then $\left(\mathcal{T}(t)\right)_{t\geq 0}$ 
is uniformly bounded. Moreover, one has
\begin{equation}
\label{eqS03}
\begin{array}{llll}
\mathcal{S}_{0.3}(t)z(x,t)
&= 0.3\ds\int_{0}^{\infty}\theta
\Phi_{_{0.3}}(\theta)\mathcal{T}(t^{0.3}\theta)z(x,t)d\theta\\
&= 0.3 \ds\int_{0}^{\infty}\theta
\Phi_{_{0.3}}(\theta)\sum_{i=1}^{\infty}e^{\upsilon_{i}
t^{0.3}\theta}\left(z,\zeta_{i}\right)\zeta_{i}(x)d\theta\\
&= 0.3\ds\sum_{i=1}^{\infty}\sum_{n=0}^{\infty}
\int_{0}^{\infty}\frac{\left(\upsilon_{i}t^{0.3}\right)^{n}}{n!}
\theta^{n+1}\Phi_{_{0.3}}(\theta)d\theta\left(z,\zeta_{i}\right)\zeta_{i}(x)\\
&= 0.3\ds\sum_{i=1}^{\infty}\sum_{n=0}^{\infty}
\frac{\left(\upsilon_{i}t^{0.3}\right)^{n}}{n!}
\frac{\Gamma(n+2)}{\Gamma(1+0.3n+0.3)}\left(z,\zeta_{i}\right)\zeta_{i}(x)\\
&= 0.3\ds\sum_{i=1}^{\infty} E_{0.3,1.3}^{2}
\left(\upsilon_{i}t^{0.3}\right)\left(z,\zeta_{i}\right)\zeta_{i}(x).
\end{array}
\end{equation}
Similarly, we have
\begin{equation}
\label{eqR03}
\mathfrak{R}_{0.3}(t)z(x,t)
= \ds\sum_{i=1}^{\infty}\left(z, \zeta_{i}\right)
E_{0.3, 1}\left(\upsilon_{i}t^{0.3}\right)\zeta_{i}(x).
\end{equation}
Let $\omega = [0.25, 0.75]$. We can easily verify that 
system \eqref{sys:exp} is $[0.25, 0.75]$-controllable 
using the same arguments in \eqref{eqomeg} and, 
by Theorem~\ref{main-res}, we get
$$
\|f\|_{F^{\circ}} = \ds\int_{0}^{\tau-h} 
\|0.3\left(\tau-\sigma-h\right)^{-0.7}\sum_{i=1}^{\infty}
E_{0.3,1.3}^{2}\left(\upsilon_{i}\left(\tau-\sigma
-h\right)^{0.3}\right)\left(z,\zeta_{i}\right)
P_{\omega}^{*}\int_{0}^{1/2}f(x)dx\|^{2}d\sigma,
$$
which defines a norm on $F^{\circ}$. Moreover, 
$\Lambda f = P_{\omega}\mathcal{P}\left(\phi_{_{2}}(\tau)\right)$ 
is an isomorphism and we obtain the control 
$$
\mathfrak{u}^{*}(t) = 0.3\left(\tau-\sigma
-h\right)^{-0.7}\ds\sum_{i=1}^{\infty}
E_{0.3, 1.3}^{2}\left(\upsilon_{i}
\left(\tau-\sigma-h\right)^{0.3}\right)\left(z, \zeta_{i}\right)
P_{\omega}^{*}\int_{0}^{1/2}f(x)dx, 
$$
steering system \eqref{sys:exp} 
from $z_{_{0}}(x)$ to $z_{_{d}}$
with minimum energy.


\subsection{Example~2: Pointwise actuator}
\label{sec:Ex2}

Let us now consider the same system as in Example~1 with the control 
operator $\mathfrak{B} = \delta(x-b)$, where $0<b<1$ is the control 
action point. The system is given by
\begin{eqnarray}
\label{sys:exp2}
\begin{cases}
\ds {}^{C}_{0}\mathbb{D}^{0.3}_{t}z(x,t)
= \Delta z(x,t) + \delta(x-b)u(t-h), 
\quad & [0,1]\times[0,\tau-h],\\
z(x,0)=0,\\
\mathfrak{u}(t) = \varphi(t), \quad -h\leq t \leq 0,
\end{cases}
\end{eqnarray}
with $\delta$ the impulse function defined by 
\begin{eqnarray}
\begin{cases}
\delta(t) = 0, \quad \text{for} \quad t\neq 0,\\
\ds\int_{-\infty}^{+\infty}\delta(t)dt = 1.
\end{cases}	
\end{eqnarray}
Here the operator $\mathfrak{B}$ is unbounded. 
Using equations \eqref{eqS03} and \eqref{eqR03}, one has
\begin{equation}
\label{eqHLphi2}
\begin{aligned}
\left(\mathcal{H} + \mathfrak{L}_{\varphi}\right)^{*} z(x,t) 
&=2\left[\mathfrak{B}^{*}(\tau-\sigma-h)^{r-1} 
\mathcal{S}_{r}^{*}(\tau-\sigma-h) z\right](x,t)\\
&=0.6 \sum_{i=1}^{\infty} \ds\frac{E_{0.3, 1.3}^{2}\left(\upsilon_{i}(\tau-\sigma-h)^{0.3}\right)
\left(z, \zeta_{i}\right)_{L^{2}(0,1)} \zeta_{i}(b)}{(\tau-\sigma-h)^{0.7}}.
\end{aligned}
\end{equation}
If $b\in\mathbb{Q}$, then system \eqref{sys:exp2} is not $\Omega$-controllable. 
Considering $\omega = [1/3,3/4]$ and $b=1/2$, system \eqref{sys:exp2} 
is $\omega$-controllable. Indeed,
\begin{equation*}
\begin{aligned}
\left(\mathcal{H} + \mathfrak{L}_{\varphi}\right)^{*} z(x,t) 
&=0.6 \sum_{i=1}^{\infty} \ds\frac{E_{0.3, 1.3}^{2}\left(\upsilon_{i}(\tau-\sigma-h)^{0.3}\right)
\left(z, \zeta_{i}\right)_{L^{2}(0,1)} \sqrt{2}\sin\left(i\ds\frac{\pi}{2}\right)}{(\tau-\sigma-h)^{0.7}}\\
&=0.6\sqrt{2} \sum_{i=1}^{\infty} \ds\frac{E_{0.3, 1.3}^{2}\left(\upsilon_{i}(\tau-\sigma-h)^{0.3}\right)
\left(z, \zeta_{i}\right)_{L^{2}(0,1)} \sin\left(i\ds\frac{\pi}{2}\right)}{(\tau-\sigma-h)^{0.7}}.
\end{aligned}
\end{equation*}
Then, 
\begin{equation*}
\begin{split}
\left(\mathcal{H}+\mathfrak{L}_{\varphi}\right)^{*}
P_{\omega}^{*}\left(P_{\omega}\zeta_{j}\right) 
& = 0.6\sqrt{2}\displaystyle\sum_{k=1}^{\infty}\frac{E_{0.3,1.3}^{2}
\left(\upsilon_{k}(\tau-\sigma-h)^{0.3}\right)}{2(\tau-\sigma-h)^{0.7}}
\langle \zeta_{i},\zeta_{j}\rangle_{L^{2}(\omega)}\sin\left(i\ds\frac{\pi}{2}\right)\\
& = 0.6\sqrt{2}\displaystyle\sum_{k=1}^{\infty}\frac{E_{0.3,1.3}^{2}
\left(\upsilon_{k}(\tau-\sigma-h)^{0.3}\right)}{2(\tau-\sigma-h)^{0.7}}\sin\left(i\ds\frac{\pi}{2}\right)
\int_{1/3}^{2/3}\zeta_{i}(x)\zeta_{j}(x)dx \\ 
& \neq 0.
\end{split}
\end{equation*}
Moreover, 
\begin{equation}
\label{normex}
\begin{array}{lll}
\|\psi_{_{0}}\|_{F^{\circ}} 
&= \ds\int_{0}^{\tau-h}\left\|(\tau-\sigma-h)^{-0.7}
 \mathcal{S}_{0.3}^{*}(\tau-\sigma-h) P_{[1/3,3/4]}^{*}\psi(b)\right\|^{2}d\sigma\\
&= \ds\int_{0}^{\tau-h}\left\|0.3\sum_{i=1}^{\infty}
\frac{E_{0.3,1.3}^{2}\left(\upsilon_{i}\left(\tau-\sigma-h\right)^{0.3}\right)
\left(z,\zeta_{i}\right) P_{\omega}^{*}
\psi(1/2)}{\left(\tau-\sigma-h\right)^{0.7}}\right\|^{2} d\sigma.
\end{array}
\end{equation}
From Lemma~\ref{lem5}, \eqref{normex} defines a norm on $F^{\circ}$ 
and \eqref{eqisomo} is an isomorphism from $F^{\circ *}$ to $F^{\circ}$, 
where $\phi_{_{2}}(\tau)$ is solution of system 
\begin{equation} 
\begin{cases}
\ds {}^{C}_{0}\mathbb{D}^{0.3}_{t}\phi_{_{2}}(t) 
= \Delta\phi_{_{2}}(t) + (\tau-t-h)^{-0.7}
\mathcal{S}_{_{0.3}}^{*}(\tau-t-h)\psi(b), \quad t\in[0,\tau-h],\\
\phi_{_{2}} (0) = 0.
\end{cases}
\end{equation}
By Theorem~\ref{main-res}, we get the control 
\begin{equation}
\mathfrak{u}^{*}(t) = \ds 0.3\sum_{i=1}^{\infty}\frac{E_{0.3,1.3}^{2}
\left(\upsilon_{i}\left(\tau-\sigma-h\right)^{0.3}\right)\left(z,\zeta_{i}\right)
P_{\omega}^{*}\psi(1/2)}{\left(\tau-\sigma-h\right)^{0.7}} d\sigma, 
\end{equation}
steering system \eqref{sys:exp2} to $z_d$, which is simultaneously 
solution of the minimization problem \eqref{sec4:min:pb}, where $\psi$ 
is a solution of \eqref{eq-reg-cont-pb}, and $\phi_{_{1}}(\tau)$ is solution of
\begin{equation}
\begin{cases}
\ds {}^{C}_{0}\mathbb{D}^{0.3}_{t}\phi_{_{1}}(t) 
= \Delta\phi_{_{1}}(t), \\
\phi_{_{1}} (0) = z_{_{0}}.
\end{cases}
\end{equation}


\section{Conclusions}
\label{sec6}

In this paper, we dealt with a fractional Caputo diffusion equation 
defined by \eqref{sys1:eq1}. We studied the regional controllability 
with delay in the control. By defining an attainable set, 
we proved the exact and weak controllability of such system. 
We also formulated a minimum optimal energy control problem 
subject to system \eqref{sys1:eq1} and computed its optimal control. 
The solution of the optimal control problem is obtained 
via an extension of the Hilbert uniqueness method.

As future work, we intend to extend the obtained results: 
(i) to the case of fractional semi-linear systems with delays 
either in the control, in the state variables or in both;
(ii) to the case of neutral evolution systems \cite{Huang22,Xi22}, 
by extending the notion of regional controllability to such systems; 
(iii) to the case of complete controllability of nonlinear 
fractional neutral functional differential equations \cite{Wen22}
with delays; (iv) to the case of regional stability 
of fractional delay systems \cite{Elshenhab2022}.
Another line of research consists to
develop the numerical part and provide 
suitable numerical simulations for real problems.
This is under investigation and will be addressed elsewhere.


\section*{Acknowledgments}

This research was funded by 
The Portuguese Foundation for Science and Technology 
(FCT---Funda\c{c}\~{a}o para a Ci\^{e}ncia e a Tecnologia) 
grant number UIDB/04106/2020 (CIDMA).



\end{document}